\documentclass{amsart}[12pt]

\usepackage{amsmath, amsthm, amscd, amsfonts,}
\usepackage{graphicx,float}
\usepackage{amssymb}
\usepackage[cmtip,arrow]{xy}
\usepackage{pb-diagram,pb-xy}

\DeclareMathOperator{\dom}{dom}

\DeclareMathOperator{\Col}{Col}
\DeclareMathOperator{\Add}{Add}

\def\MPB{{\mathbb{P}}}

\def\a{\alpha}

\newtheorem{theorem}{Theorem}[section]
\newtheorem{lemma}[theorem]{Lemma}

\newtheorem{corollary}[theorem]{Corollary}

\newtheorem{remark}[theorem]{Remark}
\newtheorem{claim}[theorem]{Claim}
\newtheorem{question}[theorem]{Question}

\numberwithin{equation}{section}

\usepackage{pb-diagram}

\def\rmark{\mbox{$\rm\bf\rule{0.06em}{1.45ex}\kern-0.05em R$}}
\def\pmark{\mbox{$\rm\bf\rule{0.06em}{1.45ex}\kern-0.05em P$}}
\def\nmark{\mbox{$\rm\bf\rule{0.06em}{1.45ex}\kern-0.05em N$}}
\def\vdash{\mbox{$\rm\| \kern-0.13em -$}}

\def\rmark{\mbox{$\rm\bf\rule{0.06em}{1.45ex}\kern-0.05em R$}}
\def\pmark{\mbox{$\rm\bf\rule{0.06em}{1.45ex}\kern-0.05em P$}}
\def\nmark{\mbox{$\rm\bf\rule{0.06em}{1.45ex}\kern-0.05em N$}}
\def\vdash{\mbox{$\rm\| \kern-0.13em -$}}
\newcommand{\lusim}[1]{\smash{\underset{\raisebox{1.2pt}[0cm][0cm]{$\sim$}}
{{#1}}}}

\title[On a question of Silver about gap-two cardinal transfer principles]{On a question of Silver about gap-two cardinal transfer principles}

\author[M. Golshani and Sh. Mohsenipour ]{Mohammad Golshani and Shahram Mohsenipour}

\date{}

\thanks{The first author's research has been supported by a grant from IPM (No. 95030417). }

\thanks{The research of the second  author was in part supported by a grant from IPM (No.
95030403).}

\begin{document}
\begin{abstract}
Assuming the existence of a Mahlo cardinal, we produce a generic extension of G\"{o}del's constructible universe $L$, in which the $GCH$ holds and the transfer principles
$(\aleph_2, \aleph_0) \rightarrow (\aleph_3, \aleph_1)$ and $(\aleph_3, \aleph_1) \rightarrow (\aleph_2, \aleph_0)$ fail simultaneously. The result answers a question
of
Silver from 1971. We also extend our result to higher gaps.
\end{abstract}

\thanks{ } \maketitle

\section{Introduction}
In this paper we study cardinal transfer principles introduced by Vaught \cite{vaught1}, \cite{vaught2}, and prove
some consistency results related to them.

Assume $\mathcal{L}$ is a first order language which contains a unary predicate $U.$ By a $(\kappa, \lambda)$-model for $\mathcal{L}$, we mean a
model $\mathcal{M}= (M, U^{\mathcal{M}}, \dots)$, where $|M|=\kappa$ and $|U^{\mathcal{M}}|=\lambda,$ where $U^{\mathcal{M}}$ is the interpretation of $U$
in $\mathcal{M}.$ Following Devlin \cite{devlin}, we use the notation 
\[
(\kappa, \lambda) \rightarrow (\kappa', \lambda')
\]
to mean the following transfer principle: 

$\hspace{1.5cm}$ For every countable first order language $\mathcal{L}$ as above, and  every  

$\hspace{1.5cm}$ first order theory $T$ of $\mathcal{L}$, if $T$ has a $(\kappa, \lambda)$-model, then it has a

$\hspace{1.5cm}$ $(\kappa', \lambda')$-model.

For any natural number $n \geq 1,$ by the \emph{gap-$n$-cardinal transfer principle} we mean the statement
\[
\forall \kappa~ \forall \lambda~ (\kappa^{+n}, \kappa) \rightarrow (\lambda^{+n}, \lambda).
\]

In \cite{silver}, Silver proved the independence of gap-2-cardinal transfer principle. Starting from an inaccessible cardinal, he was able to produce a model in which the cardinal transfer $(\aleph_3, \aleph_1) \rightarrow (\aleph_2, \aleph_0)$ fails. His proof is simply as follows:
By a result of Vaught \cite{vaught2},  there exists a sentence $\phi_{KH}$ in a suitable first order  language, such that for any infinite cardinal $\beta$,
\begin{center}
$\phi_{KH}$ has a $(\beta^{++}, \beta)$-model $\iff$ there exists a $\beta^+$-Kurepa tree.
\end{center}
Now, starting from an inaccessible cardinal $\kappa$, Silver shows that in the generic extension by the Levy collapse $\Col(\aleph_1, < \kappa),$ there are no $\aleph_1$-Kurepa trees. If we start with $V=L,$ then in the resulting extension, there are $\aleph_2$-Kurepa trees, and so
 the transfer principle  $(\aleph_3, \aleph_1) \rightarrow (\aleph_2, \aleph_0)$ fails in it. Similarly if we force with $\Col(\aleph_2, < \kappa)$, then in the extension there are no $\aleph_2$-Kurepa trees, and we can use it to prove the independence of  $(\aleph_2, \aleph_0) \rightarrow (\aleph_3, \aleph_1)$. The following question was asked by Silver \cite{silver}.
\begin{question} \footnote{On page 388 of \cite{silver}, Silver writes ``One can also get a $GCH$ model in which $(\aleph_7, \aleph_5) \rightarrow (\aleph_3, \aleph_1)$ fails
and a $GCH$ model which $(\aleph_3, \aleph_1) \rightarrow (\aleph_7, \aleph_5)$ fails (though I don't see how to get the $\rightarrow$ both ways to fail simultaneously)''.}
Is it consistent with $GCH$ that both  transfer principles $(\aleph_3, \aleph_1) \rightarrow (\aleph_2, \aleph_0)$  and $(\aleph_2, \aleph_0) \rightarrow (\aleph_3, \aleph_1)$ fail simultaneously$?$
\end{question}

\begin{remark}
If we drop the $GCH$ assumption from the question, then one can easily answer the above question. Assume $\kappa$
is an inaccessible cardinals and let $G \ast H$ be $\Col(\aleph_1, < \kappa) \ast \lusim{\Add}(\aleph_0, \kappa)$-generic over $L$.
In the generic extension $L[G \ast H]$ there are no $\aleph_1$-Kurepa trees (see Devlin \cite{devlin1}) but  there exists an $\aleph_2$-Kurepa tree, and hence by the remarks above,
the transfer principle $(\aleph_3, \aleph_1) \rightarrow (\aleph_2, \aleph_0)$
fails in $L[G \ast H].$

On the other hand $L[G \ast H]$ satisfies ``$2^{\aleph_0}=2^{\aleph_1}= \kappa=\aleph_2$''.
Let $\mathcal{L}=(U, F)$, where $U$ is a unary predicate symbol and $F$ is a binary predicate symbol.
let $T$ be an $\mathcal{L}$-theory which says the following:
\begin{enumerate}
\item $\forall x,y ~  F(x, y) \to U(y)$. In particular, for each $x$, $F$ determines a subset $F_x$ of $U,$ namely, $F_x=\{y: F(x, y)       \}.$
\item For all $x \neq x', F_x \neq F_y$.
\end{enumerate}
Then $T$ has an $ (\aleph_2, \aleph_0)$ model but it does not have an $ (\aleph_3, \aleph_1)$-model (as otherwise we should have $2^{\aleph_1} \geq \aleph_3$).
Thus the transfer principle $(\aleph_2, \aleph_0) \rightarrow (\aleph_3, \aleph_1)$
fails in $L[G \ast H].$

\end{remark}
We give an affirmative answer to this question by proving the following  theorem:
\begin{theorem} \label{silver problem}
Assume $\kappa$ is a Mahlo cardinal. Then there is a generic extension of $L$, the G\"{o}del's constructible universe, in which the $GCH$ holds and the cardinal transfer
principles
 $(\aleph_{2}, \aleph_0) \rightarrow (\aleph_3, \aleph_1)$  and $(\aleph_3, \aleph_1) \rightarrow (\aleph_{2}, \aleph_0)$ fail.
\end{theorem}
Then we prove a general model theoretic fact, and use it to extend the above result to higher gaps:
\begin{theorem} \label{extended silver problem}
Assume $\kappa$ is a Mahlo cardinal. Then there is a generic extension of $L$ in which the $GCH$ holds and for all $n \geq 2,$ the cardinal transfer principles
 $(\aleph_{n}, \aleph_0) \rightarrow (\aleph_{n+1}, \aleph_1)$  and $(\aleph_{n+1}, \aleph_1) \rightarrow (\aleph_{n}, \aleph_0)$ fail.
\end{theorem}
\begin{remark}
Our proofs can be easily extended to get the following consistency result: assume $\alpha < \beta$ are regular cardinals and assume there exists a Mahlo cardinal above them. Then in a generic extension of $L$, the $GCH$ holds and both transfer principles $(\alpha^{+n}, \alpha) \rightarrow (\beta^{+n}, \beta)$ and $(\beta^{+n}, \beta) \rightarrow (\alpha^{+n}, \alpha)$
fail.
\end{remark}

 In Section \ref{sec:silver problem} we prove Theorem \ref{silver problem} and in Section \ref{extended silver}, we prove Theorem \ref{extended silver problem}. In the last section, we discuss the same problem for the case of gap-1.

\section{Proof of Theorem \ref{silver problem}}
In this section we prove Theorem \ref{silver problem}.
\label{sec:silver problem}
\subsection{On a result of Jensen}
\label{sec:jensen work}
In this subsection we state a result of Jensen \cite{jensen} and mention some of its basic properties which are needed.
Let
$\mathcal{L}=\{\in, A, \mathcal{C}                  \},$
where $A$ is a unary predicate and $\mathcal{C}$ is a function symbol.
Let $T_J$ be the following theory in $\mathcal{L}$:
\begin{center}
``$ZFC^-$$+GCH+ A^+$ is the largest cardinal$+ \mathcal{C}$ is a $\Box_{A^+}$-sequence''.
\end{center}
By a $(\kappa, \lambda)$-model of $T_J$ we mean a model $\mathcal{M}=(M, \in^{\mathcal{M}}, A^{\mathcal{M}}, \mathcal{C}^{\mathcal{M}})$ of $T_J$,
where $|M|=\kappa$ and $|A^{\mathcal{M}}|=\lambda.$
\begin{theorem}
\label{jensen theorem}
(Jensen \cite{jensen})
Assume $GCH+\Diamond_{\beta^+}$ holds, where $\beta$ is a regular cardinal, and suppose  $\kappa > \beta$ is a Mahlo cardinal. Then there is a forcing notion $\MPB_{\beta, \kappa}$ such that if
$K$ is $\MPB_{\beta, \kappa}$-generic
over $V$, then the following hold in $V[K]$:
\begin{enumerate}
\item [(a)] $V[K] \models$``$GCH$''.

\item [(b)] The principle $\Diamond^+_{\beta^+}$ holds.

\item [(c)] The theory $T_J$ does not have any $(\beta^{++}, \beta)$-model.
\end{enumerate}
\end{theorem}
\begin{proof}
As requested by the referees, we sketch the proof of the theorem, by providing the forcing construction $\MPB_{\beta, \kappa}$,  and refer to \cite{jensen} for details.
Let $G$ be $\Col(\beta^+, <\kappa)$-generic over $V$, where
$$\Col(\beta^+, <\kappa)=\{p: \beta^+ \times \kappa \to \kappa: |p| \leq \beta \text{~and for all~}(\alpha, \lambda) \in \dom(p),~p(\alpha, \lambda) < \lambda            \}$$
is the Levy collapse. The next claim is standard.
\begin{claim}
\label{claim levy collapse}
\begin{itemize}
\item [(a)] The forcing $\Col(\beta^+, <\kappa)$ is $\beta^+$-closed and $\kappa$-c.c.

\item [(b)] $V[G] \models$``$GCH+\Diamond_{\beta^+}$''.

\item [(c)]  $V[G] \models$``$\kappa=\beta^{++}$ and $\Box_{\beta^{++}}$ fails''.
\end{itemize}
\end{claim}
In \cite{jensen}, the following strengthening of Claim \ref{claim levy collapse}$(c)$
is proved.
\begin{claim}
\label{claim more levy collapse}
In $V[G],$ the theory $T_J$ has no $(\beta^{++}, \beta)$-model
\end{claim}
From now on we work in $V[G].$ Let $\mathcal{S}=\langle  S_\a: \a < \beta^+     \rangle$
witness $\Diamond_{\beta^+}.$ For each $\alpha < \beta^+$
let $d_\alpha: \beta \to \alpha$ be an onto function and set $d= \langle  d_\alpha: \alpha < \beta^+   \rangle$.
For
$\alpha < \beta^+$ set
\[
M_\alpha= L_{\gamma_\alpha}[\mathcal{S} \upharpoonright \alpha+1, d \upharpoonright \alpha+1],
\]
where $\gamma_\alpha$ is the least ordinal $\gamma > \alpha$
such that $\gamma > \sup_{\nu < \alpha} \gamma_\nu$ and
$$L_{\gamma}[\mathcal{S} \upharpoonright \alpha+1, d \upharpoonright \alpha+1] \models~``ZFC^-\text{~''}.$$
Define
\[
\mathcal{S}^*= \langle   S^*_\alpha: \alpha < \beta^+       \rangle,
\]
where $S^*_\alpha= P(\alpha) \cap M_\alpha.$
We find a generic extension of $V[G]$ in which $\mathcal{S}^*$
is a $\Diamond^+_{\beta^+}$-sequence.

Let $A \subseteq \kappa$ be such that $L_\kappa[A]=H(\kappa)$
and define the sequence $\langle \rho_\nu: \nu < \kappa  \rangle  $ by recursion on $\nu$
as follows: $\rho_\nu$ is the least ordinal $\rho > \beta^+$ such that
\begin{itemize}
\item $\rho > \sup_{\xi < \nu} \rho_\xi.$
\item $\langle M_\alpha: \alpha < \beta^+  \rangle \in L_\rho[A].$
\item $cf(\rho)=\beta^+.$
\item $L_\rho[A] \models$``$ZFC^-+ \forall x, |x| \leq \beta^+.$
\end{itemize}
Set $\tilde{\rho}_\nu = \beta^+ \cup \sup_{\xi < \nu}\rho_\xi$,
\[
\mathcal{U}_\nu = \langle  L_{\rho_\nu}[A], \in, A \cap \rho_\nu,   \langle M_\alpha: \alpha < \beta^+  \rangle                   \rangle,
\]
and for $\nu> 0$ set
\[
\tilde{\mathcal{U}}_\nu = \bigcup_{\xi < \nu}\tilde{\mathcal{U}}_\xi = \langle  L_{\tilde{\rho}_\nu}[A], \in, A \cap \tilde{\rho}_\nu,   \langle M_\alpha: \alpha < \beta^+  \rangle                   \rangle.
\]
Then set

$\hspace{2.5cm}$ $f_\nu=$ the $\mathcal{U}_\nu$-least bijection $f: \beta^+ \leftrightarrow \tilde{\rho}_\nu$.

$\hspace{2.5cm}$ $a_\xi=$ the $\xi$-th  $a \subseteq \beta^+$ in $L_\kappa[A].$

$\hspace{2.5cm}$ $\tilde{a}_\nu = \{ (\xi, \mu): \xi \in a_{f_{\nu}(\mu)}    \}.$

We are now ready to define the desired forcing notion, that we denote  by $Add(\Diamond^+_{\beta^+})$. First we define the forcing notions $Add(\Diamond^+_{\beta^+})_\nu, \nu < \kappa,$ which are the building blocks of the main forcing construction \footnote{In \cite{jensen}, the forcing notion $Add(\Diamond^+_{\beta^+})_\nu$ is denoted by $\MPB_\nu^A$ and the forcing notion $Add(\Diamond^+_{\beta^+})$ is denoted by $\MPB^A$}.

A condition in $Add(\Diamond^+_{\beta^+})_\nu$ is a subset $p$ of $\beta^+$ such that
\begin{enumerate}
\item $p \subseteq \beta^+$ is closed and bounded.
\item  $\alpha \in p \implies \tilde{a}_\nu \cap \alpha \in M_\alpha.$
\end{enumerate}
$Add(\Diamond^+_{\beta^+})_\nu$ is ordered by end extension:
\[
p \leq q \iff q = p \cap (\max(p)+1).
\]
Let us now define $Add(\Diamond^+_{\beta^+}).$ A condition in $Add(\Diamond^+_{\beta^+})$ is a function $p$
such that
\begin{enumerate}
\item $\dom(p) \subseteq \kappa$ and $|\dom(p)| \leq \beta$.
\item  $\forall \nu \in \dom(p), p(\nu) \in Add(\Diamond^+_{\beta^+})_\nu.$
\item If $\nu \in \dom(p),$ then
\begin{enumerate}
\item $f_\nu''[\max(p(\nu))] \subseteq \dom(p).$
\item For each $\xi \in f_\nu''[\max(p(\nu))], ~ \max(p(\xi)) \geq \max(p(\nu))$.
\item $\alpha \in p(\nu) \implies \tilde{C}_{p, \nu} \cap \alpha \in M_\alpha,$ where
\[
\tilde{C}_{p, \nu} = \{  (\mu, \xi) \in \max(p(\nu)) \times \max(p(\nu)): \mu \in p(f_{\nu}(\xi))             \}.
\]
\end{enumerate}
\end{enumerate}
The forcing $Add(\Diamond^+_{\beta^+})$ is ordered as follows: $p \leq q$ if and only if
\begin{center}
$\dom(p) \supseteq \dom(q)$
and for all $\nu \in \dom(q), p(\nu) \leq_{Add(\Diamond^+_{\beta^+})_\nu} q(\nu)$.
\end{center}
Let $H$ be $Add(\Diamond^+_{\beta^+})$-generic over $V[G].$
The next claim is proved in \cite{jensen}.
\begin{claim}
\label{jensen main claim}
\begin{enumerate}
\item [(a)] $Add(\Diamond^+_{\beta^+})$ is $\beta^+$-distributive and $\kappa=\beta^{++}$-c.c.''.

\item [(b)] $V[G \ast H]\models$``$GCH$''.

\item [(c)] $\mathcal{S}^*$ witnesses that $\Diamond^+_{\beta^+}$ holds in $V[G \ast H].$

\item [(d)] The theory $T_J$ does not have a $(\beta^{++}, \beta)$-model in $V[G \ast H].$
\end{enumerate}
\end{claim}
Then $\MPB_{\beta, \kappa}= \Col(\beta^+, <\kappa) \ast \lusim{Add}(\Diamond^+_{\beta^+})$ is as required.
\end{proof}
Suppose $K=G \ast H$ is $\MPB_{\beta, \kappa}$-generic over $V$.
As $\Diamond^+_{\beta^+}$ implies the existence of a $\beta^+$-Kurepa tree \cite{devlin},  in $V[K],$ we have $\beta^+$-Kurepa trees.

\subsection{Completing the proof of Theorem \ref{silver problem}}
\label{sec:completing the proof}
In this subsection we complete the proof of Theorem \ref{silver problem}. Thus assume $V=L$ and let $\kappa$ be a Mahlo cardinal. Let $\lambda$ be the least inaccessible cardinal. So $\lambda < \kappa.$ Let $G$ be $\Col(\aleph_1, < \lambda)$-generic over $L$. Then:
\begin{lemma}
\label{extension by levy collapse}

\begin{itemize}

\item [(a)] $L[G]\models$ ``There are no $\aleph_1$-Kurepa trees''.
\item [(b)] $L[G]\models$ `` $GCH$ holds''.
 \item [(c)] $L[G]\models$ `` $\kappa$ is a Mahlo cardinal''.
\end{itemize}
\end{lemma}
\begin{proof}
$(a)$ and $(b)$ hold by \cite{silver}, and $(c)$ is clear, as the forcing $\Col(\aleph_1, < \lambda)$ has size $< \kappa$.
\end{proof}
Let $K$ be $\MPB^{L[G]}_{\aleph_1, \kappa}$-generic over $L[G]$.
 We show that $L[G \ast K]$ is the required model. First note that by
 Theorem \ref{jensen theorem},
 \begin{center}
$L[G \ast K] \models$`` there exists an $\aleph_2$-Kurepa tree''.
\end{center}
But by Lemma \ref{extension by levy collapse},  $L[G]\models$ ``There are no $\aleph_1$-Kurepa trees''.
On the other hand, $L[G] \models$``$\MPB_{\aleph_1, \kappa}$ is $\lambda=\aleph_2$-distributive'',
in particular
 \begin{center}
$L[G \ast K]\models$ ``There are no $\aleph_1$-Kurepa trees''.
 \end{center}
It follows that
\[
L[G \ast K] \models \text{~``~}(\aleph_3, \aleph_1) \rightarrow (\aleph_2, \aleph_0) \text{~fails ''}.
\]
On the other hand, by Theorem  \ref{jensen theorem}(b),
$L[G \ast K] \models$``$T_J$ does not have an $(\aleph_3, \aleph_1)$-model''.
We show that $T_J$ has an $(\aleph_2, \aleph_0)$-model in $L[G \ast K]$. First note that $\aleph_2^{L[G \ast K]}=\lambda,$ which is inaccessible but not Mahlo in
$L$, so it follows from results of Jensen and Solovay (see \cite{devlin}) that
$\Box_{\aleph_1}$ holds in both $L[G]$ and $L[G \ast K]$. Let $\mathcal{C}=\langle C_\alpha: \alpha < \lambda, \lim(\alpha) \rangle \in L[G]$
witness this.
Consider the model
\[
\mathcal{M}=(H(\lambda)^{L[G]}, \in, \aleph_0, \mathcal{C}),
\]
where $\aleph_0$ is considered as the interpretation of $A.$
Then $\mathcal{M}$ is an  $(\aleph_2, \aleph_0)$-model of $T$. So
\[
L[G \ast K] \models \text{~``~}(\aleph_2, \aleph_0) \rightarrow (\aleph_3, \aleph_1) \text{~fails ''}.
\]
The theorem follows.

\section{A general model theoretic fact and the proof of Theorem \ref{extended silver problem}}
\label{extended silver}
In this section we prove a general model theoretic fact, and use it to prove Theorem \ref{extended silver problem}.
\subsection{A general model theoretic fact} In this subsection we prove the following lemma and consider some of its consequences.
\begin{lemma}
\label{a model theoretic fact}
Assume $n \geq 1,$ $\mathcal{L}$ is a first order language which contains a unary predicate $U,$
and $T$ is a theory in $\mathcal{L}$.  Then there are $\mathcal{L}^+ \supseteq \mathcal{L}$ and a theory $T^+$ in
$\mathcal{L}^+$, such that for all infinite cardinals $\beta$:
\begin{center}
$T$ has a $(\beta^{+n}, \beta)$-model $\iff$ $T^+$ has a $(\beta^{+n+1}, \beta)$-model.
\end{center}
\end{lemma}
\begin{proof}
Let $\mathcal{L}^+ = \mathcal{L} \cup \{<, W_0, \dots, W_n, F_{-1}, F_0, \dots, F_n              \}$
where $<$ is a binary predicate symbol, $W_i$'s are unary predicate symbols, $F_{-1}$ is a binary predicate symbol and $F_i$'s, $0 \leq i \leq n,$ are ternary
predicate symbols.
Let $T^+$ consists of the following axioms:
\begin{enumerate}
\item $\phi^{W_n},$ for each $\phi \in T,$ where $\phi^{W_n}$ is the relativization of $\phi$ to $W_n$.
\item $<$ is a linear ordering of the universe.
\item Under $<$, each $W_i$ is an initial segment of $W_{i+1}, i<n,$ and $W_n$ is an initial segment of the universe (in particular $W_0 \subseteq W_1 \subseteq
\dots \subseteq W_n$).
\item $U \subseteq W_n$ (i.e., $\forall x (U(x) \to W_n(x))$).
\item $F_{-1} \subseteq U \times W_0$ defines a bijection from $U$ onto $W_0$.
\item For each $0 \leq i < n, F_i \subseteq (W_{i+1}\setminus W_i) \times W_i \times W_{i+1}$ is such that if $x \in W_{i+1} \setminus W_i,$
then $\{(y, z): F_i(x, y, z)          \}$
is a bijection from $W_i$ onto $\{z \in W_{i+1}: z < x           \}$.
\item $F_n$ is such that if $x \notin W_n,$ then $\{(y, z): F_n(x, y, z)          \}$
is a bijection from $W_n$ onto $\{z: z < x           \}$.
\end{enumerate}
Now suppose that $T$ has a $(\beta^{+n}, \beta)$-model $\mathcal{M} = (\beta^{+n}, U^{\mathcal{M}}, \dots)$. Consider the model
\[
\mathcal{M}^+ = (\beta^{+n+1}, \mathcal{M}, <, \beta,  \dots, \beta^{+n}, f_{-1}, f_0, \dots, f_n),
\]
where $f_{-1}: U^{\mathcal{M}} \leftrightarrow \beta,$ each $f_i, 0 \leq i \leq n$
is such that for each $\beta^{+i} \leq \gamma < \beta^{+i+1}, \{(\zeta, \eta): (\gamma, \zeta, \eta) \in f_i            \}$
defines a bijection $\beta^{+i} \leftrightarrow \gamma.$
It is easily seen that $\mathcal{M}^+$ is a $(\beta^{+n+1}, \beta)$-model for $T^+$.

Conversely assume that $\mathcal{M}^+$ is a $(\beta^{+n+1}, \beta)$-model for $T^+$. Consider the model
$\mathcal{M}$ which is obtained from $\mathcal{M}^+ \upharpoonright \mathcal{L}$, by replacing its universe with
$W_n^{\mathcal{M}^+}$. It follows from $(1)$ that $\mathcal{M}$
is a model of $T$. We show that it is a  $(\beta^{+n}, \beta)$-model. We have $U^{\mathcal{M}} = U^{\mathcal{M}^+},$
which has size $\beta.$ On the other hand, axioms $(4)$-$(6)$ can be used to show that $|W_0^{\mathcal{M}^+}|=\beta,$
$|W_{i+1}^{\mathcal{M}^+}| \leq |W_i^{\mathcal{M}^+}|^+$
 and $|W_m^{\mathcal{M}^+}| \geq \beta^{+n}$, so by induction on $i \leq n$, we have $|W_i^{\mathcal{M}^+}|=\beta^{+i}.$
In particular $|W_n^{\mathcal{M}^+}|=\beta^{+n},$ and the result follows.
\end{proof}
\begin{corollary}
\label{transfer implications}
For each $n \geq 1,$ the gap-$(n+1)$-cardinal transfer principle implies the gap-$n$-cardinal transfer principle.
\end{corollary}
\begin{remark}
\label{enayat remark}
In personal communication, Ali Enayat informed us that Corollary 3.2 is an immediate consequence of the downward L\"{o}wenheim-Skolem theorem, i.e., the fact that if $\mathcal{M}=(M, \dots)$ is an infinite structure in a countable language and $X$ is any subset of $M$, then there is an elementary substructure $\mathcal{M}_0=(M_0, \dots)$ of $\mathcal{M}$ that includes $X$ and whose cardinality is  $\max \{\aleph_0, |X|\}$.  Using this theorem, it is easy to see that every model $\mathcal{M}$ that exhibits a gap-$m$ model, say $(\kappa^{+m}, \kappa),$ for some $m>0$ has an elementary sub-model $\mathcal{M}_0$ that exhibits a gap-$n$ model $(\kappa^{+n}, \kappa)$ for all $n<m.$
\end{remark}
\subsection{Proof of Theorem \ref{extended silver problem}}
In this subsection we complete the proof of Theorem \ref{extended silver problem}.  Let $L[G \ast H]$
be the model obtained in Subsection \ref{sec:completing the proof}. So in $L[G \ast H]$
both transfer principles $(\aleph_3, \aleph_1) \rightarrow (\aleph_2, \aleph_0)$
and $(\aleph_2, \aleph_0) \rightarrow (\aleph_3, \aleph_1)$ fail. So, by induction, and using Lemma \ref{a model theoretic fact}, for each $n \geq 2,$ the transfer principles
$$(\aleph_{n}, \aleph_0) \rightarrow (\aleph_{n+1}, \aleph_1)$$
  and
$$(\aleph_{n+1}, \aleph_1) \rightarrow (\aleph_{n}, \aleph_0)$$
  fail in $L[G \ast H].$

\section{The case of gap-$1$ and some problems}
In general, we can not hope to prove a result as above for gap-$1$-cardinal transfer principles. This is because of Vaught's theorem \cite{vaught2} that the transfer principle
$(\beta^+, \beta) \rightarrow (\aleph_1, \aleph_0)$ is a theorem of $ZFC$. However we do not know the answer to the following question:
\begin{question}
 Is it consistent that both transfer principles $(\aleph_2, \aleph_1) \rightarrow (\aleph_3, \aleph_2)$ and
 $(\aleph_3, \aleph_2) \rightarrow (\aleph_2, \aleph_1)$ fail simultaneously.
\end{question}
As we showed in Corollary \ref{transfer implications}, the gap-$(n+1)$-cardinal
 transfer principle implies the gap-$n$-cardinal transfer principle.

On the other hand if $L[G]$ is a generic extension of $L$ by the Levy collapse of an inaccessible cardinal $\kappa$ to $\aleph_2,$
then it follows from results of Vaught \cite{vaught2}, Chang \cite{chang} and Jensen \cite{devlin}
that the gap-$1$-cardinal transfer principle holds in $L[G]$, while by Silver's result stated in the introduction,
the gap-$2$-cardinal transfer principle fails in $L[G]$. We do not know the answer for higher gaps.
\begin{question}
Assume $n>1$. Is it consistent that the gap-$n$-cardinal
 transfer principle holds while the  gap-$(n+1)$-cardinal
 transfer principle fails$?$
\end{question}

\subsection*{Acknowledgements}

The authors would like to thank the referees of the paper for their very helpful comments and corrections. They also thank Ali Enayat for
his interest in this work and his Remark \ref{enayat remark}.

School of Mathematics, Institute for Research in Fundamental Sciences (IPM), P.O. Box:
19395-5746, Tehran-Iran.

E-mail address: golshani.m@gmail.com

School of Mathematics, Institute for Research in Fundamental Sciences (IPM), P.O. Box:
19395-5746, Tehran-Iran.

E-mail address: sh.mohsenipour@gmail.com

\end{document}